\documentclass[reqno, 12pt]{amsart}
\usepackage{amsmath,amsthm,amsfonts,amssymb}
\usepackage[cp1251]{inputenc}
   \usepackage[english]{babel}
\date{}
\newtheorem{theorem}{Theorem}%[section]
\newtheorem{corollary}{Corollary}%[section]
\newtheorem{proposition}{Proposition}%[section]
\begin{document}

\centerline{{\bf Vladimir I. Bogachev}}

\vskip .2in

\centerline{{\bf Compactification of spaces of measures and pseudocompactness}}

\vskip .2in

Abstract. We prove pseudocompactness of a Tychonoff space $X$ and the space $\mathcal{P}(X)$ of Radon probability measures on it with the
weak topology under the condition that the Stone--\v{C}ech compactification of the space $\mathcal{P}(X)$ is homeomorphic
to the space $\mathcal{P}(\beta X)$ of Radon probability measures on the Stone--\v{C}ech compactification of the space~$X$.

\vskip .1in



\vskip .1in

Let $X$ be a a Tychonoff  (i.e., completely regular) space, let $\beta X$ be its Stone-\v{C}ech  compactification, and let
$\mathcal{P}(X)$ be the space of Radon probability measures on $X$ equipped with the weak topology.
In the recent paper \cite{B24}, the question about coincidence of the space $\mathcal{P}(\beta X)$ of Radon probability
measures on $\beta X$ with the  Stone--\v{C}ech compactification $\beta \mathcal{P}( X)$ of the
space $\mathcal{P}(X)$.
This coincidence is understood in the following sense: the extension by continuity  of the natural embedding
$\mathcal{P}( X)\to \mathcal{P}(\beta X)$ onto the compactification is one-to-one.
The following result has been proved  in the cited paper.

\begin{theorem}
{\rm(i)} Pseudocompactness of the space of measures $\mathcal{P}(X)$ implies pseudocompactness of~$X$.

{\rm(ii)}  The injectivity of the indicated extension of the embedding $\mathcal{P}( X)\to \mathcal{P}(\beta X)$ onto
$\beta \mathcal{P}( X)$ implies
pseudocompactness of both spaces~$X$ and~$\mathcal{P}(X)$.
\end{theorem}

On the other hand,  examples of noncompact spaces~$X$ are constructed in ~\cite{B24}  for which the indicated coincidence holds.
Such spaces include the open interval of countable ordinals  $[0,\omega_1)$
and the deleted Tychonoff planck $[0,\omega_0]\times [0,\omega_1]\backslash (\omega_0,\omega_1)$ (where $\omega_0$ is the smallest countable
ordinal and  $\omega_1$ is the first uncountable ordinal), which shows  that
coincidence is possible without countable compactness. As noted in~\cite{B24}, if in place of the Stone--\v{C}ech  compactifications we take
the Samuel compactifications with suitable uniformities, then  coincidence always holds.
Finally, a number of questions has been posed in \cite{B24} in connection with compactifications of spaces of measures.
In this paper, an answer is given to one of these questions and the main result from \cite{B24} is reinforced, namely, it  is shown that the
 spaces $\mathcal{P}(X)$ and $X$ are pseudocompact if the spaces $\mathcal{P}(\beta X)$  and $\beta \mathcal{P}( X)$ are homeomorphic.
In particular, the spaces  $\mathcal{P}(\beta \mathbb{N})$ and $\beta  \mathcal{P}(\mathbb{N})$ are not homeomorphic.
 However, in the general case the question about equality of the  spaces $\mathcal{P}(\beta X)$ and $\beta \mathcal{P}( X)$
  in the sense explained above provided they are homeomorphic remains open. No characterization of spaces for which the equality holds is known.

 For a Tychonoff space $X$, we denote   by $C_b(X)$ the set of all bounded continuous functions on~$X$
and by $\beta X$ its Stone--\v{C}ech  compactification (a compact space in which $X$ is embedded homeomorphically as an everywhere
 dense subset with the property that every function from  $C_b(X)$ is the restriction to $X$ of a function from
 $C_b(\beta X)$). The linear space of all bounded Radon measures on $X$ is denote by $\mathcal{M}(X)$.
We recall that a signed measure $\mu$ on the Borel
   $\sigma$-algebra of the space $X$ is called Radon if  its  positive and negative parts are Radon,
and a nonnegative Borel measure $\mu$ is Radon if, for every Borel set $B$ and every $\varepsilon>0$,
there is a  compact set $K\subset B$ such that $\mu(B\backslash K)<\varepsilon$.

   The weak topology on the space $\mathcal{M}(X)$ is defined by the seminorms
   $$
   p_f(\mu)=\biggl|\int_X f\, d\mu\biggr|, \quad f\in C_b(X).
   $$
The   space $\mathcal{P}(X)$ of  Radon probability measures is equipped with the induced  weak topology.

The coincidence of $\beta\mathcal{P}(X)$ and $\mathcal{P}(\beta X)$ in the indicated sense is equivalent to the property that every
bounded continuous function on $\mathcal{P}(X)$ is uniformly approximated by functions of the form
$$
\mu\mapsto F(l_{f_1}(\mu),\ldots,l_{f_n}(\mu)), \quad l_{f_i} (\mu)=\int_X f_i\, d\mu, \quad f_i\in C_b(X),
$$
where $F$ is a polynomial on $\mathbb{R}^n$ (see~\cite{B24}).

The following result reinforces  assertion~(ii) of Theorem~1 from \cite{B24} with a shorter justification
(in the part concerning pseudocompactness of $\mathcal{P}(X)$, but the conclusion about  $X$ itself is based on assertion~(i)
of Theorem~1).

\begin{theorem}
Suppose that $X$ is a Tychonoff  space  such that  the spaces $\beta\mathcal{P}(X)$
and $\mathcal{P}(\beta X)$ are homeomorphic. Then $\mathcal{P}(X)$ and $X$ are pseudocompact.
\end{theorem}
\begin{proof}
According to \cite[Theorem~2]{B24} it suffices to verify pseudocompactness of the space $\mathcal{P}( X)$.
Now we apply the following known fact (see \cite{Banas}, \cite{Hen} or \cite[p.~221, Theorem~9.3]{W}): if a Tychonoff space a possesses locally
connected Stone-\v{C}ech compactification, then is it pseudocompact. Recall that
a space is called locally connected if every point possesses a base of connected open neighborhoods.
This fact is applied to the space $\mathcal{P}(X)$, the compactification of which in the considered case
is locally connected, since by assumption it is homeomorphic to the convex compact set $\mathcal{P}(\beta X)$ in the
 locally convex space $\mathcal{M}(X)$ of Radon measures on $X$ equipped with the weak
topology. Of course, the conclusion of the theorem remains valid under the formally weaker condition of local connectedness of the
compactification of the space $\mathcal{P}(X)$, but one can hardly regard this condition simpler that the direct requirement of pseudocompactness
of~$\mathcal{P}(X)$.
\end{proof}

\begin{corollary}
Suppose that a Tychonoff space $X$ can be continuously mapped onto an everywhere dense set
in a space~$Y$ that is not  pseudocompact (for example, in a noncompact metric or Souslin space).
Then the spaces $\beta\mathcal{P}(X)$ and $\mathcal{P}(\beta X)$ are not homeomorphic.

In particular, this is true if $X$ is a noncompact metric or Souslin space.
\end{corollary}
\begin{proof}
Under this condition  the space $X$ is not pseudocompact, since the composition of an unbounded continuous function on
the space $Y$ with the given mapping from $X$ to $Y$ is also continuous and unbounded.
\end{proof}

It is known  (see \cite{Colmez} or \cite[Theorem~1.2.2]{AA}) that a pseudocompact  space is Baire, i.e., the
intersection of every sequence of open everywhere dense sets is everywhere dense.
Hence we obtain the following conclusion.

\begin{corollary}
If the spaces  $\beta\mathcal{P}(X)$ and $\mathcal{P}(\beta X)$ are homeomorphic, then  $\mathcal{P}(X)$ and $X$ are Baire spaces.
\end{corollary}

In turn, this implies the Baire property of all powers $X^n$ (see \cite{Koum}, the case of metric spaces was
considered in~\cite{Woj}). Though, in our case this  also follows from  pseudocompactness of~$X$
(see, e.g., \cite[Theorems 5.3.2 and 5.3.3]{Dor}).
Generally speaking, the Baire property of $X$ does not imply the Baire property of $\mathcal{P}(X)$ even in the case of metric spaces,
but it would be interesting to clarify whether the Baire property of $\mathcal{P}(X)$ follows from pseudocompactness of~$X$.
Properties of spaces of measures connected with the Baire property were studied in \cite{Brown} and \cite{BrownCox}.
Note that in \cite{Krup} an example is constructed in ZFC  (the Zermelo--Fraenkel set theory with the axiom of choice)
of a separable metric space~$X$ that is not
Polish, but all closed subsets of  $\mathcal{P}(X)$ are Baire spaces.

There are two other concepts connected with the name of Baire:
the Baire $\sigma$-algebra generated by continuous functions on a Tychonoff  space $X$  and
measures on this $\sigma$-algebra, called Baire measures.
The space of Baire measures is  also equipped with the weak topology.
Let $\mathcal{P}_\sigma(X)$ denote the space of  Baire probability measures with the weak topology.
Pseudocompactness of $X$ is equivalent to compactness of $\mathcal{P}_\sigma(X)$. Hence in the case where
$\beta\mathcal{P}(X)$ and $\mathcal{P}(\beta X)$ are homeomorphic and  every Baire measure on~$X$ has a Radon extension
we obtain the usual compactness of~$X$, not just  pseudocompactness.

The equality $\beta\mathcal{P}(X)=\mathcal{P}(\beta X)$ implies  that the Baire  $\sigma$-algebra
of the space $\mathcal{P}(X)$ is also generated by the linear functionals of the form
$$
\mu\mapsto \int_X f\, d\mu, \quad f\in C_b(X).
$$
This follows from the fact that, once the indicated equality holds, every bounded continuous function on $\mathcal{P}(X)$
is uniformly approximated by polynomials in such functionals (as already noted above).
The Baire  $\sigma$-algebra of the whole  space of measures $\mathcal{M}(X)$ is always  generated by such functionals, see
\cite[Theorem~6.10.6]{B07}, but it is not clear whether this is true for $\mathcal{P}(X)$. The set
$\mathcal{P}(X)$ is Baire in $\mathcal{M}(X)$ under a rather restrictive condition, as the next
assertion shows.

\begin{proposition}
The set $\mathcal{P}(X)$ belongs to the  Baire  $\sigma$-algebra of the space  $\mathcal{M}(X)$  precisely when
there is a countable collection of continuous functions on~$X$ separating points, i.e., $X$ can be continuously and injectively mapped into~$\mathbb{R}^\infty$.
\end{proposition}
\begin{proof}
Suppose that there is a countable collection of continuous functions on~$X$ separating points. We can pass to a countable collection of bounded
functions with this property. To this collection we  add all polynomials with rational coefficients in finitely many given functions,
and then substitutions of the obtained functions into the function $t\mapsto \min(\max(t,0),1)$.
This gives a countable collection $\{f_n\}\subset C_b(X)$. Let us show that the set $\mathcal{P}(X)$ is determined in $\mathcal{M}(X)$
as the intersection of the sets of the form $\{l_f(\mu)\ge 0\}$ and $\{l_1(\mu)=1\}$, where $l_f(\mu)$ is the integral against the measure $\mu$
of a nonnegative function from the collection~$\{f_n\}$ and $l_1(\mu)=\mu(X)$. It suffices to verify that
the set of the indicated inequalities determines the set
of nonnegative measures. Suppose that there is a signed measure $\mu$ the integrals against  which of nonnegative functions from $\{f_n\}$
are nonnegative.
We can assume that $|\mu|(X)\le 1$, where $|\mu|$ is the total variation of the measure~$\mu$.
Let $K$ be a compact set for which $\mu(K)=-c<0$ and $\mu^{+}(K)=0$, where $\mu^{+}$ is the nonnegative part of~$\mu$.
Let us take a larger compact set $S$ with $\mu^{+}(X\backslash S)<c/4$ and also an open set $U$ with $K\subset U$ and $\mu^{+}(U)<c/4$.
Since $X$ is completely regular, there exists a continuous function $g$ on $X$ with $0\le g\le 1$, $g|_K=1$ and $g|_{X\backslash U}=0$.
 Due to our  choice of the collection $\{f_n\}$ and the Stone--Weierstrass theorem this collection contains a function $f_n$ such that
 $0\le f_n\le 1$ and $\sup_S |g(x)-f_n(x)|<c/4$. Then taking into account the estimates
 $|\mu|(S)\le 1$,  $|\mu|(X\backslash S)<c/4$ and $|\mu|(U\backslash K)<c/4$ we obtain
 $$
 \int_X f_n\, d\mu\le \int_S f_n\, d\mu +c/4\le
   \int_S g\, d\mu +c/2=\mu(K)+\int_{U\backslash K} g\, d\mu \le \mu(K)+ 3c/4=-c/4,
 $$
  which gives a contradiction.

Suppose now that $\mathcal{P}(X)$ belongs to the Baire $\sigma$-algebra in $\mathcal{M}(X)$. Then it has the form
$$
P=\{\mu\in \mathcal{M}(X)\colon (F_1(\mu), F_2(\mu),\ldots)\in B\},
$$
where $B$ is a Borel set in $\mathbb{R}^\infty$,
$$
F_i(\mu)=\int_X f_i\, d\mu,\quad f_i\in C_b(X).
$$
Observe that the functions $f_i$ separate points of the  space~$X$. Indeed, otherwise there are two
different points $a,b\in X$ such that $f_i(a)=f_i(b)$ for all~$i$.
Hence $F_i(\delta_a)=F_i(\delta_a + (\delta_a-\delta_b))$ for all~$i$. Since $\delta_a\in P$,
we obtain  $\delta_a + (\delta_a-\delta_b)\in P$, which is false.
\end{proof}

Some equivalent conditions to the existence of a sequence of continuous functions separating points can be found
in \cite{KoumS84} (see also \cite[Theorem~8.10.39]{B07}).

In  \cite{B24}, the example $X=\beta\mathbb{N}\backslash \{p\}$ was considered, where $p\in \mathbb{N}^*:=\beta \mathbb{N}\backslash \mathbb{N}$.
Then $X$ is countably compact and $\beta X=\beta \mathbb{N}$  (see \cite[p.~239, Exercise~58]{AP} or \cite[Section~2.14]{Vaug}), and if $p$
is a  $P$-point, then the space $\mathcal{P}(X)$ is also countably compact, as shown in \cite{B24}.
A point $p\in \mathbb{N}^*$ is called a $P$-point if the intersection of every countable collection of its  open neighborhoods contains its
open neighborhood in~$\mathbb{N}^*$. Such points exist under the continuum hypothesis (CH)
or under its negation and Martin's axiom (MA), but their absence is consistent with ZFC,
see  \cite[Exercise~55 in Chapter~IV]{AP}, \cite[p.~138]{GillJer}, \cite[Corollary~1.7.2]{Mill},
\cite[p.~107]{W},  \cite{Sz}, and~\cite{Wimmers}).
Since the Stone--\v{C}ech compactification is one-point here, the space $\mathcal{P}(\beta X)$ coincides with the convex span
of the set $\mathcal{P}(X)$ and the Dirac measure $\delta_p$ at the point~$p$.
However, it is not clear whether it coincides with $\beta\mathcal{P}(X)$.

It remains unknown whether the equality $\beta\mathcal{P}(X)=\mathcal{P}(\beta X)$ follows from pseudocompactness of $X$ or $\mathcal{P}(X)$,
and also from countable compactness of one of these spaces.
It is unclear whether it is preserved by multiplication by compact spaces (recall  that pseudocompactness has this property,
see~\cite{AA}). In the examples constructed in \cite{B24}  when this equality holds for noncompact spaces their Stone--\v{C}ech
 compactifications are one-point. It is of interest to consider the case of a general space $X$ with a one-point Stone--\v{C}ech
compactification  (such a  space is automatically pseudocompact, see \cite[Theorem~1.3.8]{AA}).
A~known particular case  is the Mr\'owka
space  (see  \cite[Exercise~3.6.I]{Engel} or a detail discussion in~\cite{AA}),
for which $M=S\cup D$, where $S$ is countable and everywhere dense in~$M$, all points of $S$ are isolated, and $D$ is an uncountable
discrete closed set disjoint with~$S$.
Here $\beta M=M\cup \{p\}$, $\mathcal{P}(M)$ consists of probability measures concentrated on countable sets,
$\mathcal{P}(\beta M)$ coincides with the convex span of $\mathcal{P}(M)$ and the Dirac measure at the point~$p$ of the one-point
compactification.
Another example worth of consideration is the subspace $S$ in the Tychonoff cube $[0,1]^{[0,1]}$ consisting
of functions with countable support (the so-called $\Sigma$-product).
This subspace is sequentially compact and everywhere dense and $\beta S=[0,1]^{[0,1]}$, see
\cite[Section~2.7.13]{Engel}.

The question is also open about possible noncoincidence of the cardinalities of $\beta\mathcal{P}(X)$ and $\mathcal{P}(\beta X)$.
Since $\mathcal{P}(\beta X)$ is a convex compact set in the locally convex space of Radon measures,
its cardinality $\kappa$ coincides with its  countable power, i.e.,  $\kappa=\kappa^{\omega_0}$, see~\cite{Lip}.

I thank K.A.\,Afonin for useful discussions.

The paper is supported by the Ministry of Science and Education of Russia within the programme of the Moscow Center of
Fundamental and Applied Mathematics, agreement N 075-15-2022-284.

%This research is supported by the Russian Science Foundation Grant 22-11-00015 (at Lomo\-no\-sov Moscow State University).

Lomonosov Moscow State university, Moscow, Russia;
National Research University ``Higher School of Economics'', Moscow, Russia.

\end{document}